\documentclass{amsart}
\usepackage{amsfonts}
\usepackage{graphicx}
\usepackage{epstopdf}
\usepackage{hyperref}
\usepackage{multirow}
\usepackage[table]{xcolor}
\newcolumntype{R}[1]{>{\raggedleft\arraybackslash }b{#1}}
\newcolumntype{L}[1]{>{\raggedright\arraybackslash }b{#1}}
\newcolumntype{C}[1]{>{\centering\arraybackslash }b{#1}}

\newtheorem{theorem}{Theorem}[section]
\theoremstyle{plain}

\newtheorem{case}{Case}

\newtheorem{corollary}{Corollary}

\newtheorem{proposition}{Proposition}
\newtheorem{remark}{Remark}

\numberwithin{equation}{section}

\begin{document}
\title{Uniform in bandwidth consistency  for the transformation kernel estimator of copulas}
\author{Cheikh Tidiane Seck}
\author{Diam Ba}

\author{Gane Samb Lo}

\date{\today}

\begin{abstract}
In this paper we establish the uniform in bandwidth consistency for the transformation kernel estimator of copulas introduced in \cite{r6}. To this end, we  prove a uniform in bandwidth law of the iterated logarithm for the maximal deviation of this estimator from its expectation. We then show that, as $n$ goes to infinity, the bias of the estimator converges to zero uniformly in the bandwidth $h$ varying over a suitable interval. A practical method of selecting the optimal bandwidth is presented. Finally, we make conclusive simulation experiments, showing the performance of the estimator on finite samples.\\

\end{abstract}

\maketitle

\section{Introduction}
Let $(X_1,Y_1),...,(X_n,Y_n)$ be an independent and identically distributed random sample of a random vector $(X,Y)$, with joint cumulative distribution function $H$ and marginal distribution functions $F$ and $G$.  Let $K(\cdot,\cdot)$ represent a multiplicative kernel distribution function ; i.e., $K(x,y)=K(x)K(y)$ and $0<h_n<1$ denote a bandwidth sequence. The transformation kernel estimator of copulas introduced in \cite{r6} is defined as follows :
\begin{equation}\label{te}
\hat{C}_{n}^{(T)}(u,v)=\frac{1}{n}\sum_{i=1}^{n}K\left(\frac{\phi^{-1}(u)-\phi^{-1}(\hat{U}_i)}{h_n} \right)K\left(\frac{\phi^{-1}(v)-\phi^{-1}(\hat{V}_i)}{h_n} \right),
\end{equation}
where $\phi$ is an increasing transformation and $\hat{U}_i$, $\hat{V}_i$ are pseudo-observations. It is customary in coplua  estimation to take $\hat{U}_i=\frac{n}{n+1}F_n(X_i)$, $\hat{V}_i=\frac{n}{n+1}G_n(Y_i)$, where $F_n$ and $G_n$ are the empirical marginal cumulative distribution functions. This estimator presents an advantage comparatively to the estimator proposed by Fermanian \textit{et al.} \cite{r3}(2004), as it does not depend on the marginal distributions. Taking $\phi$ equal to the standard Gaussian distribution leads to the  Probit transformation proposed, for instance, in Marron and Ruppert \cite{mr} (1994). For nonparametric kernel estimation for the copula density using the Probit transformation, we refer to Geenens \textit{et al.} \cite{gcp} (2014) and references therein.  \\

In this paper we are concerned with kernel estimation for the copula function, avoiding as such the inconsistency problem due to many unbounded copula densities. However, there is a boundary bias problem which may be solved by using the transformation kernel estimator \eqref{te}, with a suitable bandwidth. Since the choice of the bandwidth is problematic for $\hat{C}_{n}^{(T)}(u,v)$, as pointed out in \cite{r6}, we shall deal with a variable bandwidth $h$ that may depend either on the data or the location point $(u,v)$. Thus, we define the following estimator :
\begin{equation}\label{e1}
\hat{C}_{n,h}^{(T)}(u,v)=\frac{1}{n}\sum_{i=1}^{n}K\left(\frac{\phi^{-1}(u)-\phi^{-1}(\hat{U}_i)}{h} \right)K\left(\frac{\phi^{-1}(v)-\phi^{-1}(\hat{V}_i)}{h} \right).
\end{equation}
We shall assume that $K(.)$ is the integral of a symmetric bounded kernel $k(.)$ supported on $[-1,1]$ satisfying the following conditions :
\begin{itemize}
\item[(K.1)] $\int_{-1}^{1}k(s)ds=1$ ;
\item[(K.1)] $k(.)$ is a 2-order kernel ; i.e., $\int_{-1}^{1}sk(s)ds=0$ and $\int_{-1}^{1}s^2k(s)ds\neq 0$.
\end{itemize} 
 Our main goal is to establish the strong consistency of $\hat{C}_{n,h}^{(T)}(u,v)$ uniformly in $h$ varying over a suitable interval $[a_n, b_n],\,0< a_n\leq b_n<1$. These results enable us to apply  various methods of bandwidth selection and obtain the consistency of  estimators, under  certain conditions on $h$.\\
 
 The rest of the paper is organized as follows. In Section 2, we state our main theoreticla results and give their proofs. In Section 3, we present a practical method for seclecting the bandwith, which is based on a cross-validation criterion. In Section 4, we make a simulation study using data generated with the Frank copula. Finally, the paper is ended by an Appendix. 
 \section{Main results }\label{ssec2}
We state our theoretical results in this section 
\begin{theorem}\label{t1}
Suppose that the copula function $C(u,v)$ has bounded first-order partial derivatives on $(0, 1)^2$ and 
 the transformation $\phi$ admits a bounded derivative $\phi'$. Then, for any sequence of positive constants $(b_n)_{n\geq 1}$ satisfying $0<b_n<1, b_n\rightarrow 0$ and $b_n\geq (\log n)^{-1}$, we have almost surely, for some $c>0$, as $n\rightarrow\infty$ 
\begin{equation}\label{e2}
R_n\sup_{\frac{c\log n}{n}\le h\le b_n}\sup_{(u,v)\in(0,1)^2}\left|\hat{C}_{n,h}^{(T)}(u,v)-\mathbb{E}\hat{C}_{n,h}^{(T)}(u,v)\right|=O(1),
\end{equation}
where $R_n =\left(\frac{n}{2\log\log n}\right)^{1/2}$.
\end{theorem}
 
 \begin{theorem} \label{t2}
Suppose that the copula function  $C(u,v)$  has bounded second-order partial  derivatives on $(0,1)^2$ and that the  transformation $\phi$ admits a bounded derivative $\phi'$. Then, for any sequence of positive constants $(b_n)_{n\geq 1}$ satisfying $0<b_n<1$ and $\sqrt{n}b_n^2/\sqrt{\log\log n}=o(1),$  we have almost surely, for some $c>0$, as $n\rightarrow\infty$,
\begin{equation} \label{biais}
R_n\sup_{0< h\le b_n}\sup_{(u,v)\in(0,1)^2}\vert \mathbb{E}\hat{C}_{n,h}^{(T)}(u,v) - C(u,v)\vert=o(1),
\end{equation}
where $R_n=\left(\frac{n}{2\log\log n}\right)^{1/2}$.
\end{theorem}
The following proposition is an immediate consequence of Theorem \ref{t1} and Theorem \ref{t2}
\begin{proposition}\label{pp1}
Let $a_n=c\log n/n$ for some  $c>0$ and $0<b_n<1$ such that $\sqrt{n}b_n^2/\sqrt{\log\log n}=o(1).$ Then, under the assumptions of theorems \ref{t1} and \ref{t2}, we have almost surely, as $n\rightarrow\infty$,
\begin{equation} \label{cv}
\sup_{a_n\le h\le b_n}\sup_{(u,v)\in(0,1)^2}\vert \hat{C}_{n,h}^{(T)}(u,v) - C(u,v)\vert \rightarrow 0.
\end{equation}
\end{proposition}

 \begin{proof} (\textbf{Theorem \ref{t1}})
We begin by some notation. Recall that $H_n$, $F_n$ and $G_n$ are the empirical cumulative distribution functions of $H$, $F$ and $G$, respectively. Then  the copula estimator based directly on Sklar's Theorem can be defined as 
$$ C_n(u,v)= H_n(F_n^{-1}(u),G_n^{-1}(v)),$$
with $F_n^{-1}(u)=\inf\{x:F_n(x)\geq u\}$ and $G_n^{-1}(v)=\inf\{x:F_n(x)\geq v\}$ the quantile functions corresponding to $F_n$ and $G_n$. Define the bivariate empirical copula process as
$$\mathbb{C}_n(u,v)=\sqrt{n}[C_n(u,v)- C(u,v)],\qquad (u,v)\in [0,1]^2$$
and introduce the following quantity.
$$ \widetilde{C}_n(u,v)=\frac{1}{n}\sum_{i=1}^n\mathbb{I}\{U_i\leq u,V_i\leq v\}$$
which represents the uniform bivariate empirical distribution function based on a sample $(U_1,V_1),\cdots,(U_n,V_n)$ of independent and identically distributed random  variables with marginals uniformly distributed on $[0,1]$. 
Define the following empirical process
 $$\mathbb{\widetilde{C}}_n(u,v)=\sqrt{n}[\widetilde{C}_n(u,v)- C(u,v)],\quad (u,v)\in[0,1]^2.$$
Then, one can easily prove that  
\begin{equation}\label{c3}
\mathbb{\widetilde{C}}_n(u,v)=\mathbb{C}_n(u,v)+ \frac{1}{\sqrt{n}}.
\end{equation}
Let $\phi$ be an increasing transformation with values in $[0,1]$. For $n\geq 1$, $0<h<1$, set
$$D_{n,h}(u,v):=\hat{C}_{n,h}^{(T)}(u,v)-\mathbb{E}\hat{C}_{n,h}^{(T)}(u,v)$$ and 
$$g_{n,h}:=\hat{C}_{n,h}^{(T)}(u,v)-\widetilde{C}_n(u,v). $$ Then, one has
\begin{eqnarray*}
g_{n,h}&=& \frac{1}{n}\sum_{i=1}^{n}\left [K\left(\frac{\phi^{-1}(u)-\phi^{-1}(\hat{U}_i)}{h}\right)K\left(\frac{\phi^{-1}(v)-\phi^{-1}(\hat{V}_i)}{h} \right)- \mathbb{I}\{U_i\leq u,V_i\leq v\}\right ]\\
& =& \frac{1}{n}\sum_{i=1}^{n}\left [K\left(\frac{\phi^{-1}(u)-\phi^{-1}(\hat{F}_n \circ F^{-1}(U_i))}{h}\right)K\left(\frac{\phi^{-1}(v)-\phi^{-1}(\hat{G}_n\circ G^{-1}(V_i))}{h} \right)-\mathbb{I}\{U_i\leq u,V_i\leq v\}\right]\\
& =:& \frac{1}{n}\sum_{i=1}^{n}g(U_i,V_i,h),
\end{eqnarray*}
where $g$ belongs to the class of measurable functions $\mathcal{G}$ defined as 
$$
\mathcal{G}=\left\{\begin{array}{c} 
g:(s,t,h)\mapsto g(s,t,h)=K\left(\frac{\phi^{-1}(u)-\phi^{-1}(\zeta_{1,n}(s))}{h}\right) K\left(\frac{\phi^{-1}(v)-\phi^{-1}(\zeta_{2,n}(t))}{h}\right)- \mathbb{I}\{s\leq u,t\leq v\},\\
 u,v\in[0,1], 0<h<1 \,\text{and }\, \zeta_{1,n} ; \zeta_{2,n}:[0,1]\mapsto [0,1] \,\text{nondecreasing.}
\end{array} \right\}
$$
Since $\mathbb{E}\widetilde{C}_n(u,v)=C(u,v)$, one can observe that 
$$
\sqrt{n}\vert g_{n,h}- \mathbb{E}g_{n,h}\vert= \vert\sqrt{n}D_{n,h}(u,v)-\mathbb{\widetilde{C}}_n(u,v)\vert.
$$

Now, we have to apply the main Theorem of  Mason and Swanepoel (2010) \cite{r5} which gives the order of convergence of the deviation from their expectations of kernel-type function estimators. Towards this end, the above class of functions $\mathcal{G}$ must satisfy the following  four conditions : 
\begin{itemize}
\item[(G.i)]There exists a finite constant $\kappa>0$ such that
 $$ \sup_{0\leq h\leq 1}\sup_{g\in \mathcal{G}}\left\|g\left(\cdot,\cdot,h\right)\right\|_\infty=\kappa <\infty. $$
\item[(G.ii)]\ \ \ There exists a constant $C'>0$ such that for all $h\in [0,1]$,
$$ \sup_{g\in \mathcal{G}}\mathbb{E}\left[g^2\left(U,V,h\right)\right]\leq C'h. $$
\item[(F.i)]\ \ \ 
$\mathcal{G}$ satisfies the uniform entropy condition, i.e., 
$$\exists \, C_0>0, \nu_0>0\ :\ N\left(\epsilon,\mathcal{G}\right)\leq C_0\epsilon^{-\nu_0}.$$
\item[(F.ii)]\ \ \ $\mathcal{G}$ is a pointwise measurable class, i.e there exists a countable sub-class $\mathcal{G}_0$ of $\mathcal{G}$ such that for all $g\in \mathcal{G}$, there exits $\left(g_m\right)_m\subset \mathcal{G}_0$ such that $g_m\longrightarrow g.$\\
\end{itemize}

The checking of these conditions will be done in Appendix and constitutes the proof of the following proposition.
\begin{proposition}\label{p1}
Suppose that the copula function $C$ has bounded first-order partial derivatives on $(0, 1)^2$ and that the transformation $\phi$ admits a bounded derivative $\phi'$. Then assuming (G.i), (G.ii), (F.i) and (F.ii), we have for some $c > 0,\ 0 < h_0 < 1,$ with probability one, 
$$
\limsup_{n\rightarrow\infty}\sup_{\frac{c\log n}{n}\leq h \leq h_0}\sup_{(u,v)\in(0,1)^2}
\frac{|\sqrt{n}D_{n,h}(u,v)-\tilde{\mathbb{C}}_n(u,v)|}{\sqrt{h(|\log h\vert\vee\log\log n)}}=A(c),
$$
where $A(c)$ is a positive constant.
\end{proposition}

\begin{corollary}\label{crl1}
 Under the assumptions of Proposition \ref{p1}, one has for any sequence of constants $0<b_n<1,$ satisfying $\ b_n\rightarrow 0,\ b_n\geq (\log n)^{-1}$, with probability one, 
 $$
 \sup_{\frac{c\log n}{n}\leq h \leq b_n}\sup_{(u,v)\in(0,1)^2}
\frac{|\sqrt{n}\hat{D}_{n,h}^{(T)}(u,v)-\tilde{\mathbb{C}}_n(u,v)|}{\sqrt{\log\log n}}=O(\sqrt{b_n}).
 $$
 \end{corollary}
\begin{proof}{( \textbf{Corollary \ref{crl1})}}\\
{\rm First, observe that the condition $b_n\geq (\log n)^{-1}$ implies
\begin{equation}\label{dc}
\frac{\vert \log b_n\vert}{\log\log n}\leq 1.
\end{equation}
Next, by the monotonicity of the function $x\mapsto x\vert\log x \vert$ on $[0,1/e]$, one can write for $n$ large enough, $h\vert\log h \vert\leq b_n\vert\log b_n \vert$ and hence,
\begin{equation}
h(\vert\log h \vert  \vee\log\log n)\leq b_n(\vert\log b_n \vert \vee\log\log n).
\end{equation}
Combining this and Proposition \ref{p1}, we obtain 
$$
 \sup_{\frac{c\log n}{n}\leq h \leq b_n}\sup_{(u,v)\in(0,1)^2}
\frac{|\sqrt{n}\hat{D}_{n,h}^{(T)}(u,v)-\tilde{\mathbb{C}}_n(u,v)|}{\sqrt{b_n\log\log n\left(\frac{\vert \log b_n\vert}{\log\log n}\vee 1\right)}}=O(1).
 $$
Thus the Corollary \ref{crl1} follows from  \eqref{dc}.}\\
\end{proof}

Coming back to the proof of our Theorem \ref{t1}, we have to show that the deviation $D_{n,h}(u,v)$, suitably normalized, is almost surely uniformly bounded, as $n\rightarrow\infty$. For this, it suffices to prove that
\begin{equation}\label{lil}
\limsup_{n\rightarrow\infty}\sup_{\frac{c\log n}{n}\leq h \leq b_n}\sup_{(u,v)\in[0,1]^2}\frac{\left|\sqrt{n}D_{n,h}(u,v)\right|}{\sqrt{2\log\log n}}\leq 3.
\end{equation}
We will make use of an approximation of the empirical copula process $\mathbb{C}_n$ by a Kiefer process (see e.g., Zari\cite{r11}, page 100). Let $\mathbb{W}(u,v,t)$ be a $3$-parameters  Wiener process defined on $[0,1]^2\times[0,\infty)$. Then the Gaussian process $\mathbb{K}(u,v,t)=\mathbb{W}(u,v,t)-\mathbb{W}(1,1,t).uv$ is called a $3$-parameters  Kiefer process defined on $[0,1]^2\times[0,\infty)$.\\
\indent By Theorem 3.2 in Zari\cite{r11}, for $d=2$, there exists a sequence of Gaussian processes $\left\{\mathbb{K}_{C}(u,v,n), u,v\in[0,1], n>0\right\}$ such that
$$ \sup_{(u,v)\in[0,1]^2}\left|\sqrt{n}\mathbb{C}_n(u,v)-\mathbb{K}_{C}^\ast(u,v,n)\right|=O\left(n^{3/8}(\log n)^{3/2}\right),$$ where 
$$\mathbb{K}_{C}^\ast(u,v,n)=\mathbb{K}_{C}(u,v,n)-\mathbb{K}_{C}(u,1,n)\frac{\partial C(u,v)}{\partial u}-\mathbb{K}_\mathbb{C}(1,v,n)\frac{\partial C(u,v)}{\partial v}.$$
This yields
\begin{equation}\label{c1}
\limsup_{n\rightarrow\infty}\sup_{(u,v)\in[0,1]^2}\frac{\left|\mathbb{C}_n(u,v)\right|}{\sqrt{2\log\log n}}=\limsup_{n\rightarrow\infty}\sup_{(u,v)\in[0,1]^2}\frac{\left|\mathbb{K}_{C}^\ast(u,v,n)\right|}{\sqrt{2n\log\log n}}.
\end{equation}
By the works of Wichura\cite{r10} on the  law of the iterated logarithm , for $d=2$, one has almost surely
\begin{equation}\label{c2}
\limsup_{n\rightarrow\infty}\sup_{(u,v)\in[0,1]^2}\frac{\left|\mathbb{K}_\mathbb{C}^\ast(u,v,n)\right|}{\sqrt{2n\log\log n}}\leq 3,
\end{equation}

\noindent which entails  
$$\limsup_{n\rightarrow\infty}\sup_{(u,v)\in[0,1]^2}\frac{\left|\mathbb{C}_n(u,v)\right|}{\sqrt{2\log\log n}}\leq 3.$$
Since ${\mathbb{C}}_n(u,v)$ and $\tilde{\mathbb{C}}_n(u,v)$ are asymptotically equivalent in view of (\ref{c3}), one obtains
$$\limsup_{n\rightarrow\infty}\sup_{(u,v)\in[0,1]^2}\frac{\left|\tilde{\mathbb{C}}_n(u,v)\right|}{\sqrt{2\log\log n}}\leq 3.$$
Applying Corollary \ref{crl1} and recalling the fact that $b_n\rightarrow 0$, one obtains \eqref{lil} which proves Theorem \ref{t1}.
\end{proof}
\begin{proof} (\textbf{Theorem \ref{t2}})
 Let 
$$ B_{n,h}(u,v) =  \mathbb{E}\hat{C}_{n,h}^{(T)}(u,v) - C(u,v).$$

Observe that by  hypothesis (H.1) on the kernel $k(.)$, we can write for all $(u,v)\in [0,1]^2$,
$$
C(u,v)=\int_{-1}^{1}\int_{-1}^{1}C(u,v)k(s)k(t)dsdt.
$$
Put $\hat{U}_i=\frac{n}{n+1}F_n(X_i)=\frac{n}{n+1}F_n\circ F^{-1}(U_i)=:\zeta_{1,n}(U_i))$  and $\hat{V}_i=\frac{n}{n+1}G_n(Y_i)=\frac{n}{n+1}G_n\circ G^{-1}(V_i)=:\zeta_{2,n}(V_i))$ . Then, we can write
\begin{eqnarray*}
\mathbb{E}\hat{C}_{n,h}^{(T)}(u,v)&= &\mathbb{E}\left[K\left(\frac{\phi^{-1}(u)-\phi^{-1}(\hat{U_i})}{h} \right)K\left(\frac{\phi^{-1}(v)-\phi^{-1}(\hat{V_i})}{h}\right)\right]\\
&=& \int_{-1}^{1}\int_{-1}^{1}\mathbb{E}\mathbb{I}\{U_i\leq\zeta_{1,n}^{-1}[\phi(\phi^{-1}(u)-sh)],V_i\leq \zeta_{2,n}^{-1}[\phi(\phi^{-1}(v)-th )]\}k(s)k(t)dsdt\\
&=& \int_{-1}^{1}\int_{-1}^{1}C\left(\zeta_{1,n}^{-1}[\phi(\phi^{-1}(u)-sh )],\zeta_{2,n}^{-1}[\phi(\phi^{-1}(v)-th)]\right)k(s)k(t)dsdt.
\end{eqnarray*}
Thus
\begin{equation}
B_{n,h}(u,v) = \int_{-1}^{1}\int_{-1}^{1}\left[ C\left(\zeta_{1,n}^{-1}[\phi(\phi^{-1}(u)-sh)],\zeta_{2,n}^{-1}[\phi(\phi^{-1}(v)-th )]\right) - C(u,v)\right] k(s)k(t)dsdt.
\end{equation}
Making use of the Chung (1949)'s law of the iterated logarithm, we can infer that, whenever $F$ is continuous and admits a bounded density, for all $u\in [0,1]$, as $n\rightarrow\infty$,
$$\zeta_{1,n}^{-1}(u)-u \approx F\circ\hat{F}_n^{-1}(u)-F\circ F^{-1}(u)=O(n^{-1}\log\log n).$$
That is, $\zeta_{1,n}^{-1}(u)$ is asymptotically equivalent to $u$. As well, we have $\zeta_{2,n}^{-1}(v)=G\circ\hat{G}_n^{-1}(v)$ is asymptotically equivalent to $v$, for all $v\in [0,1]$ . Thus, for all large $n$, one can write
$$
B_{n,h}(u,v)= \int_{-1}^{1}\int_{-1}^{1}\left[ C(\phi(\phi^{-1}(u)-sh),\phi(\phi^{-1}(v)-th)) - C(u,v)\right] k(s)k(t)dsdt.
$$
By applying a 2-order Taylor expansion for the copula function $C$, we obtain
\begin{eqnarray*}
& C(\phi(\phi^{-1}(u)-sh),\phi(\phi^{-1}(v)-th)) -C(u,v)= &\\
& [\phi(\phi^{-1}(u)-sh)-u]C_u(u,v) + [\phi(\phi^{-1}(v)-th)-v]C_v(u,v)+ [\phi(\phi^{-1}(u)-sh)-u]^2\frac{C_{uu}(u,v)}{2} & \\
& + [\phi(\phi^{-1}(v)-th)-v]^2\frac{C_{vv}(u,v)}{2}+ [\phi(\phi^{-1}(u)-sh)-u][\phi(\phi^{-1}(v)-th)-v] C_{uv}(u,v)+o(h^2),& 
\end{eqnarray*}
where $$C_{uu}(u,v)=\frac{\partial^2 C}{\partial u^2}(u,v)\; ;\;C_{vv}(u,v)=\frac{\partial^2 C}{\partial v^2}(u,v)\; ;\;C_{uv}(u,v)=\frac{\partial^2 C}{\partial u\partial v}(u,v). $$
Applying again a 1-order Taylor expansion for the function $\phi$, we get
$$ \phi(\phi^{-1}(u)-sh)-u=\phi(\phi^{-1}(u)-sh) - \phi(\phi^{-1}(u))=-\phi'(\phi^{-1}(u))sh+o(h)$$
and
$$ \phi(\phi^{-1}(v)-sh)-v=\phi(\phi^{-1}(v)-th) - \phi(\phi^{-1}(v)) =-\phi'(\phi^{-1}(v))th+o(h).$$
Thus
\begin{eqnarray*}
& C(\phi(\phi^{-1}(u)-sh),\phi(\phi^{-1}(v)-th)) -C(u,v)= &\\
& -\phi'(\phi^{-1}(u))sh C_u(u,v)  -\phi'(\phi^{-1}(v))th C_v(u,v)+ o(h)\\
& [\phi'(\phi^{-1}(u))sh]^2\frac{C_{uu}(u,v)}{2}  +[\phi'(\phi^{-1}(v))th]^2\frac{C_{vv}(u,v)}{2}+ [\phi'(\phi^{-1}(u))][\phi'(\phi^{-1}(u))]st h^2C_{uv}(u,v)+o(h^2).& 
\end{eqnarray*}
Using the fact that $k(.)$ is 2-order kernel ; i.e., $ \int_{-1}^{1}sk(s)ds=0$ and $\int_{-1}^{1}s^2k(s)ds\neq 0,$ 
we obtain, by Fubini's Theorem, that for all $(u,v)\in [0,1]^2$,
 \begin{equation}
 B_{n,h}(u,v)=  \frac{h^2}{2}\left[ \phi'(\phi^{-1}(u))^2 C_{uu}(u,v)\int_{-1}^{1} s^2 k(s)ds + \phi'(\phi^{-1}(v))^2 C_{vv}(u,v)\int_{-1}^{1}t^2k(t)dt\right].
 \end{equation}
Since the second-order partial derivatives $C_{uu}, C_{vv}$ and $\phi'$  are assumed to be bounded,  we can write
$$
\sup_{0< h \leq b_n}\sup_{(u,v)\in[0,1]^2}B_{n,h}(u,v)= O(b_n^2).
$$
Then
\begin{equation} 
\left(\frac{n}{2\log\log n}\right)^{1/2}\sup_{0< h \leq b_n}\sup_{(u,v)\in[0,1]^2}B_{n,h}(u,v)=O\left(\frac{\sqrt{n}b_n^2}{\sqrt{2\log\log n}}\right)=o(1), 
\end{equation}
which completes the proof of Theorem \ref{t2}.
\end{proof}
\section{Bandwidth choice}
As we noted in the introduction, the choice of the bandwidth is a very difficult problem. Since the asymptotic expressions of the bias and variance of the estimator $\hat{C}_{n,h}^{(T)}(u,v)$ are not available yet, we cannot apply the plug-in method which rely on the minimization of the asymptotic \textit{mean integrated square error}. Instead, we may employ a cross-validation method following Sarda (1993)\cite{sar}.
Recall the empirical copula estimator based directly on Sklar's Theorem 
$$ C_n(u,v)= H_n(F_n^{-1}(u),G_n^{-1}(v)),$$
where $H_n$, $F_n$ and $G_n$ are the empirical cumulative distribution functions of $H$, $F$ and $G$, respectively. Let
$$\hat{C}_{n,h,-i}^{(T)}(\hat{U}_i,\hat{V}_i)=\frac{1}{n-1}\sum_{j=1,j\neq i}^{n}K\left(\frac{\phi^{-1}(\hat{U}_i)-\phi^{-1}(\hat{U}_j)}{h} \right)K\left(\frac{\phi^{-1}(\hat{V}_i)-\phi^{-1}(\hat{V}_j)}{h} \right)$$ 
be the leave-out-$(\hat{U}_i,\hat{V}_i)$ version of the estimator $\hat{C}_{n,h}^{(T)}(u,v)$ ; $\hat{U}_i$ and $\hat{V}_i$  are pseudo-observations defined previously. Then, Sarda's criterion can be defined, here, as
\begin{equation}
CV(h)=\frac{1}{n}\sum_{i=1}^{n}\left[ \hat{C}_{n,h,-i}^{(T)}(\hat{U}_i,\hat{V}_i)-C_n(\hat{U}_i,\hat{V}_i)\right]^2w(\hat{U}_i,\hat{V}_i),
\end{equation}
where $w(\cdot,\cdot)$ is a measurable bounded weight function with compact support.\\
Let $a_n$ and $b_n$ be as in Proposition \ref{pp1} and choose a data-dependent bandwidth $\hat{h}_{opt}$ that is solution to the following minimization problem :
$$ \min_{h\in [a_n,b_n]} CV(h).$$
Since $\hat{h}_{opt}\in [a_n,b_n]$, the uniform almost sure consistency of $\hat{C}_{n,\hat{h}_{opt}}^{(T)}(u,v)$ is guarranted  by Proposition \ref{pp1}.
\section{Simulation study}
Here, we  make some numerical experiments to show the performance of the transformation kernel estimator $\hat{C}_{n,\hat{h}_{opt}}^{(T)}$. Before hand, we determine graphically the optimal bandwidth $\hat{h}_{opt}$, by visualizing the curve of $CV(h)$ over $ h\in [a_n,b_n]$, where $a_n$ and $b_n$   fulfill  the conditions of Proposition \ref{pp1}. To this end, we choose $a_n=\log n/n$ and $b_n=(\log\log n/n^2)^{1/4}$ which satisfy  assumptions of Proposition \ref{pp1}, and fix  a sample size $n=100$. The interval $[a_n,b_n]$ is then equal to $[0.04,0.10]$. For simplicity, we set the weight function $w(u,v)\equiv 1$. We now consider a 0.001-valued grid of points in this interval and represent the curve of $CV(h)$ in Figure \ref{fig1}.\\
\begin{figure}[ht]
\begin{center}
\includegraphics[width=8cm]{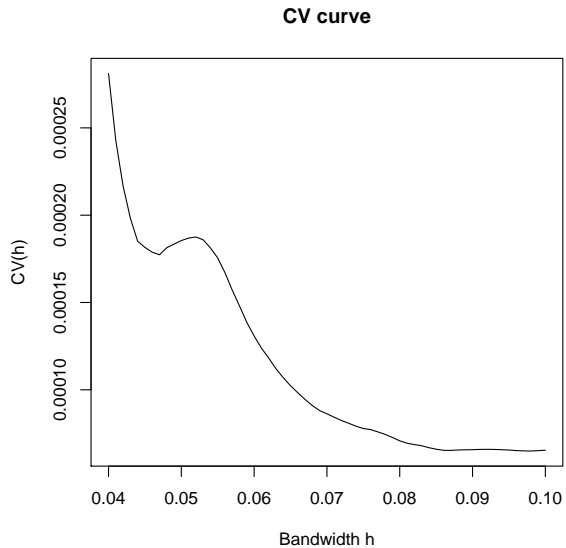}
\caption{Cross-validation curve.}
\label{fig1}
\end{center}
\end{figure}
We observe that the cross validation criterion $CV(h)$ is decreasing, when the bandwidth h increases. It is stationnary for $h\geq 0.085$ and reaches its minimal value over this grid is obtained for $h=0.085$. So, we may take $\hat{h}_{opt}=0.085$.

Next, we compute the bias and mean square error $(mse)$ of the estimator  $\hat{C}_{n,\hat{h}_{opt}}^{(T)}$ for the Frank copula, $C_{\theta}$ given below, which admits bounded second-order partial derivatives. To compute this estimator, we  employ the conditional sampling method to generate  random samples of $n$ pairs of data $(u_1,v_1),\cdots,(u_n,v_n)$ from the Frank copula, with parameter $\theta\in \mathbb{R}$, defined as 
\begin{equation}
C_{\theta}(u,v)= -\frac{1}{\theta}\log\left[ 1+ \frac{(e^{-\theta u}-1)(e^{-\theta v}-1)}{(e^{-\theta}-1)} \right].
\end{equation}
We choose the Epanechnikov kernel density $k(t)=0.75(1-t^2)\mathbb{I}(|t|\leq 1)$ to compute  the integral $K(\cdot)$. While the transformation $\phi(.)$ is taken to be the standard Gaussian distribution function, $\phi(x)=\int_{-\infty}^x e^{-t^2/2}/\sqrt{2\pi}dt$.

To estimate the $bias$ and $mse$, we generate $B=1000$ samples and  apply, for all $(u,v)\in [0,1]$, the formulas
\begin{center}
$ \displaystyle
bias(u,v)=\frac{1}{B}\sum_{b=1}^{B}\hat{C}_{n,\hat{h}_{opt},b}^{(T)}(u,v)-C_{\theta}(u,v),$\\
$ \displaystyle
mse(u,v)=\frac{1}{B}\sum_{b=1}^{B}\left(\hat{C}_{n,\hat{h}_{opt},b}^{(T)}(u,v)-C_{\theta}(u,v)\right)^2,
$
\end{center}
where $\hat{C}_{n,\hat{h}_{opt},b}^{(T)}(u,v)$ is the transformation estimation calculated with the $b^{th}$ sample.
For arbitrary values of $\theta= -2,1,5$  and different values for the couple $(u,v)$, we obtain the results in Table \ref{tab1}. In each colum of value $\theta$, we report the $bias$ first and the $mse$ below for arbitrary chosen couples $(u,v)\in(0,1)^2$. The results are very conclusive, showing that the cross validation method may be applied to select the bandwidth $h$ for the transformation kernel estimator of copulas.
\begin{table}[htbp]
\begin{center}
$$\begin{array}{|c|c|ccc|}
\hline
 (u,v) &  & \theta=-2 &  \theta=1  & \theta=5\\
 \hline
 \multirow{2}*{(0.1,0.1)}
& bias & 0.036 &  0.025  &   0.006\\
& mse & 0.0013 &  0.0006  &  4e-5\\
\hline
\multirow{2}*{(0.2,0.2)}
& bias & 0.043 &  0.008  & -0.041\\
& mse & 0.0018 &  7e-5  &  0.0017\\
\hline
\multirow{2}*{(0.5,0.5)}
& bias & 0.108 &  0.017  &   -0.078\\
& mse & 0.0117 &  0.0003  &  0.0062\\
\hline
\multirow{2}*{(0.8,0.8)}
& bias & 0.067 &  0.032  &   -0.0173\\
& mse & 0.0044 &  0.0010  &  0.0003\\
\hline
\multirow{2}*{(0.9,0.9)}
& bias & 0.057 &  0.046  &   0.027\\
& mse & 0.0033 &  0.0022  &  0.0007\\
\hline
\end{array}
$$
\end{center}
\caption{Bias and mean square error of the transformation kernel estimator.}
\label{tab1}
\end{table}

\section*{Appendix}
\begin{proof}(\textbf{Proposition \ref{p1}})\\
To simplify the notations, we consider a general function $K(\cdot,\cdot)$ which is the integral of a symmetric bounded kernel $k(\cdot,\cdot)$, supported on $[-1,1]^2$; i.e., $K(x,y)=\int_0^x\int_0^y k(s,t)dsdt$. We have to check (G.i), (G.ii), (F.i) and (F.ii).\\

\noindent \textbf{Checking for (G.i):}
Recall that $(U_i,V_i), i\geq 1$ are iid random variables uniformly distributed on $[0,1]^2$, $\zeta_{1,n}(U_i)=\hat{F}_n o F^{-1}(U_i)$ and $\zeta_{2,n}(V_i)=\hat{G}_n o G^{-1}(V_i)$.
For any function $g\in\mathcal{G}$ and $0<h<1$, we can write 
\begin{eqnarray*}
g\left(U_i,V_i,h\right) & =&K\left(\frac{\phi^{-1}(u)-\phi^{-1}(\zeta_{1,n}(U_i))}{h},\frac{\phi^{-1}(v)-\phi^{-1}(\zeta_{2,n}(V_i))}{h}\right)- \mathbb{I}\{U_i\leq u,V_i\leq v\}\\
& =&\int_{-\infty}^{\frac{\phi^{-1}(u)-\phi^{-1}(\zeta_{1,n}(U_i))}{h}} \int_{-\infty}^{\frac{\phi^{-1}(v)-\phi^{-1}(\zeta_{2,n}(V_i))}{h}}k(s,t)dsdt-\mathbb{I}\{U_i\leq u,V_i\leq v\}\\
& =&\int_{-1}^{1}\int_{-1}^{1}\mathbb{I}\left\{U_i\leq \zeta_{1,n}^{-1}\circ\phi(\phi^{-1}(u)-th) ,V_i\leq \zeta_{2,n}^{-1}\circ\phi(\phi^{-1}(v)-sh)\right\}k(s,t)dsdt \\
 & & -\mathbb{I}\{U_i\leq u,V_i\leq v\}\\
&\leq &\int_{-1}^{1}\int_{-1}^{1}k(t)k(s,t)dsdt-\mathbb{I}\{U_i\leq u,V_i\leq v\}\ \leq 4 \|k\|^2+1,\\
\end{eqnarray*}
where $\displaystyle \|k\|=\sup_{(s,t)\in[-1,1]^2}|k(s,t)|$ represents the supremum norm on $[-1,1]^2$. Thus (G.i) holds by taking $\kappa:=4\|k\|^2+1.$\\

\noindent \textbf{Checking for (G.ii).} We have to show that $\displaystyle \sup_{g\in \mathcal{G}}\mathbb{E}g^2(U,V,h)\leq C_0 h$, where $C_0$ is a positive  constant. One can write
\begin{eqnarray*}\small
&\mathbb{E}g^2(U,V,h) =  \mathbb{E}\left[K\left(\frac{\phi^{-1}(u)-\phi^{-1}(\zeta_{1,n}(U))}{h},\frac{\phi^{-1}(v)-\phi^{-1}(\zeta_{2,n}(V))}{h}\right)- \mathbb{I}\{U\leq u,V\leq v\}\right]^2 &\\
&=  \mathbb{E}\left[K^2\left(\frac{\phi^{-1}(u)-\phi^{-1}(\zeta_{1,n}(U))}{h},\frac{\phi^{-1}(v)-\phi^{-1}(\zeta_{2,n}(V))}{h}\right)\right] &\\
&  - 2\mathbb{E}\left[K\left(\frac{\phi^{-1}(u)-\phi^{-1}(\zeta_{1,n}(U))}{h},\frac{\phi^{-1}(v)-\phi^{-1}(\zeta_{2,n}(V))}{h}\right)\mathbb{I}\{U\leq u,V\leq v\}\right] + C(u,v)&\\
& =:  A -2B + C(u,v)&.
\end{eqnarray*}
 Since the function $K(\cdot,\cdot)$ is a kernel of a distribution function, we may assume without loss of generality that it takes its values in $[0,1]$. Then, we can use the inequality $K^2(x,y)\leq K(x,y)$ to bound up the term $A$ in the right hand side of the previous egality.
\begin{eqnarray*}
&A  =   \mathbb{E}\left[K^2\left(\frac{\phi^{-1}(u)-\phi^{-1}(\zeta_{1,n}(U))}{h},\frac{\phi^{-1}(v)-\phi^{-1}(\zeta_{2,n}(V))}{h}\right)\right]&\\
&\leq  \mathbb{E}\left[K\left(\frac{\phi^{-1}(u)-\phi^{-1}(\zeta_{1,n}(U))}{h},\frac{\phi^{-1}(v)-\phi^{-1}(\zeta_{2,n}(V))}{h}\right)\right]&\\
& \leq \mathbb{E}\left[ \int_{-1}^{1}\int_{-1}^{1}\mathbb{I}\left\{U\leq \zeta_{1,n}^{-1}\circ\phi(\phi^{-1}(u)-sh),V\leq \zeta_{2,n}^{-1}\circ\phi(\phi^{-1}(v)-th)\right\}k(s,t)dsdt \right].&
\end{eqnarray*}
The other term $B$ can be written into
\begin{eqnarray*}
&B  = \mathbb{E}\left[K\left(\frac{\phi^{-1}(u)-\phi^{-1}(\zeta_{1,n}(U))}{h},\frac{\phi^{-1}(v)-\phi^{-1}(\zeta_{2,n}(V))}{h}\right)\mathbb{I}\{U\leq u,V\leq v\}\right]&\\
&= \mathbb{E}\left[ \int_{-1}^{1}\int_{-1}^{1}\mathbb{I}\left\{s\leq \frac{\phi^{-1}(u)-\phi^{-1}(\zeta_{1,n}(U))}{h},t\leq \frac{\phi^{-1}(v)-\phi^{-1}(\zeta_{2,n}(V))}{h}\right\} \mathbb{I}\{U\leq u,V\leq v\}k(s,t)dsdt\right]&\\
&= \mathbb{E}\left[ \int_{-1}^{1}\int_{-1}^{1}\mathbb{I}\left\{U\leq u\wedge \zeta_1^{-1}\circ\phi(\phi^{-1}(u)-sh),V\leq v\wedge \zeta_2^{-1}\circ\phi(\phi^{-1}(v)-th)\right\}k(s,t)dsdt \right],&
\end{eqnarray*} where $x\wedge y=\min (x,y)$.
Note that $$ C(u,v)=\int_{-1}^{1}\int_{-1}^{1}C(u,v)k(s,t)dsdt,$$
as the kernel $k(\cdot,\cdot)$ satisfies  $\int_{-1}^{1}\int_{-1}^{1}k(s,t)dsdt=1$. Thus
\begin{eqnarray*}
&\mathbb{E}g^2(U,V,h)\leq \mathbb{E}\left[ \int_{-1}^{1}\int_{-1}^{1}\mathbb{I}\left\{U\leq \zeta_{1,n}^{-1}\circ\phi(\phi^{-1}(u)-sh),V\leq \zeta_{2,n}^{-1}\circ\phi(\phi^{-1}(v)-th)\right\}k(s,t)dsdt \right]&\\
 &  -2\mathbb{E}\bigg[ \int_{-1}^{1}\int_{-1}^{1}\mathbb{I}\left\{U\leq u\wedge \zeta_{1,n}^{-1}\circ\phi(\phi^{-1}(u)-sh),V\leq v\wedge \zeta_{2,n}^{-1}\circ\phi(\phi^{-1}(v)-th)\right\}k(s,t)dsdt \bigg]& \\
 & + \int_{-1}^{1}\int_{-1}^{1}C(u,v)k(s,t)dsdt.&
\end{eqnarray*}
We shall suppose that the empirical kernel distributions $\hat{F}_n$ and $\hat{G}_n$ are asymptotically equivalent to the classical empirical distribution functions $F_n$ and $G_n$, respectively. From the Chung (1949)'s LIL, we can infer that, whenever $F$ admits a bounded density, for all $u\in [0,1]$, as $n\rightarrow\infty$,
$$\zeta_{1,n}^{-1}(u)-u = F\circ\hat{F}_n^{-1}(u)-F\circ F^{-1}(u)=O(n^{-1}\log\log n).$$
That is $\zeta_{1,n}^{-1}(u)$ is asymptotically equivalent to $u$. As well, we have $\zeta_{2,n}^{-1}(v)=G\circ\hat{G}_n^{-1}(v)$ is asymptotically equivalent to $v$. Thus, for all large $n$, we can write
\begin{eqnarray*}
&\mathbb{E}g^2(U,V,h)\leq \mathbb{E}\left[ \int_{-1}^{1}\int_{-1}^{1}\mathbb{I}\left\{U\leq \phi(\phi^{-1}(u)-sh),V\leq \phi(\phi^{-1}(v)-th)\right\}k(s,t)dsdt \right]&\\
 &  -2\mathbb{E}\left[ \int_{-1}^{1}\int_{-1}^{1}\mathbb{I}\left\{U\leq u\wedge \phi(\phi^{-1}(u)-sh),V\leq v\wedge \phi(\phi^{-1}(v)-th)\right\}k(s,t)dsdt \right]&\\
 &  +\int_{-1}^{1}\int_{-1}^{1}C(u,v)k(s,t)dsdt.&
\end{eqnarray*}
That is,
\begin{eqnarray}\label{eqcas}
\mathbb{E}g^2(U,V,h)&\leq &\int_{-1}^{1}\int_{-1}^{1}C\left(\phi(\phi^{-1}(u)-sh),\phi(\phi^{-1}(v)-th)\right)k(s,t)dsdt \nonumber\\
 & & -2\int_{-1}^{1}\int_{-1}^{1}C\left(u\wedge\phi(\phi^{-1}(u)-sh),v\wedge \phi(\phi^{-1}(v)-th)\right)k(s,t)dsdt \\
 &  & +\int_{-1}^{1}\int_{-1}^{1}C(u,v)k(s,t)dsdt.\nonumber
\end{eqnarray}
Now, we have to discuss condition (G.ii) in the four following cases:\\

\textbf{Case 1.} $u\wedge\phi(\phi^{-1}(u)-sh)=\phi(\phi^{-1}(u)-sh)\ \text{et}\ v\wedge \phi(\phi^{-1}(v)-th)=\phi(\phi^{-1}(v)-th)$.\\
In this case the second member of inequality \eqref{eqcas} is reduced and we have 
\begin{eqnarray*}
\mathbb{E}g^2(U,V,h)&\leq &-\int_{-1}^{1}\int_{-1}^{1}C\left(\phi(\phi^{-1}(u)-sh),\phi(\phi^{-1}(v)-th)\right)k(s,t)dsdt \\
  &  & +\int_{-1}^{1}\int_{-1}^{1}C(u,v)k(s,t)dsdt\\
	&\leq &\int_{-1}^{1}\int_{-1}^{1}\left[C(u,v)-C\left(\phi(\phi^{-1}(u)-sh),\phi(\phi^{-1}(v)-th)\right)\right]k(s,t)dsdt.
\end{eqnarray*}
By a Taylor expansion for the copula function $C$, we have
\begin{eqnarray*}
C(u,v)-C(\phi(\phi^{-1}(u)-sh),\phi(\phi^{-1}(v)-th))&=&[u-\phi(\phi^{-1}(u)-sh)]C_u(u,v)\\
 &+ & [v-\phi(\phi^{-1}(v)-th)]C_v(u,v)+ o(h). 
\end{eqnarray*}
Applying again a Taylor-Young expansion for the function $\phi$, we obtain
$$ u-\phi(\phi^{-1}(u)-sh)=\phi(\phi^{-1}(u)) -\phi(\phi^{-1}(u)-sh)=\phi'(u)sh+o(h)$$
and
$$ v-\phi(\phi^{-1}(v)-sh)=\phi(\phi^{-1}(v)) -\phi(\phi^{-1}(v)-th)=\phi'(v)th+o(h).$$
Thus
\begin{eqnarray*}
\mathbb{E}g^2(U,V,h)&\leq &\int_{-1}^{1}\int_{-1}^{1}\left[\phi'(u)C_u(u,v)h+\phi'(v)C_v(u,v)h\right]k(s,t)dsdt \\
	&\leq & 4h\left[\|C_u\|+\|C_v\|\right]\sup_{x\in\mathbb{R}}|\phi'(x)|\|k\|.
\end{eqnarray*}
Taking $C_0=4\left[\left\|C_u\right\|+\left\|C_v\right\|\right]\left\|\phi'\right\|\left\|k\right\|$ gives condition (G.ii).\\

\textbf{Case 2.} $u\wedge\phi(\phi^{-1}(u)-sh)=u\ \text{et}\ v\wedge \phi(\phi^{-1}(v)-th)=v$. \\
Here the inequality \eqref{eqcas} is reduced to
\begin{eqnarray*}
\mathbb{E}g^2(U,V,h)\leq \int_{-1}^{1}\int_{-1}^{1}\left[C\left(\phi(\phi^{-1}(u)-sh),\phi(\phi^{-1}(v)-th)\right)-C(u,v)\right]k(s,t)dsdt.
\end{eqnarray*}
Using the same arguments as in Case 1, we obtain condition (G.ii):
$\displaystyle \sup_{g\in \mathcal{G}}\mathbb{E}g^2(U,V,h)\leq C_0h$, with $C_0=4\left[\left\|C_u\right\|+\left\|C_v\right\|\right]\left\|\phi'\right\|\left\|k\right\|$.\\

\textbf{Case 3.} $u\wedge\phi(\phi^{-1}(u)-sh)=\phi(\phi^{-1}(u)-sh)\ \text{et}\ v\wedge \phi(\phi^{-1}(v)-th)=v$.\\
Here, inequality \eqref{eqcas} is rewritten into
\begin{eqnarray*}
&\mathbb{E}g^2(U,V,h)\leq \int_{-1}^{1}\int_{-1}^{1}C\left(\phi(\phi^{-1}(u)-sh),\phi(\phi^{-1}(v)-th)\right)k(s,t)dsdt & \\
 & -2\int_{-1}^{1}\int_{-1}^{1}C\left(\phi(\phi^{-1}(u)-sh),v\right)k(s,t)dsdt  +\int_{-1}^{1}\int_{-1}^{1}C(u,v)k(s,t)dsdt &\\
&\leq \int_{-1}^{1}\int_{-1}^{1}\left[C\left(\phi(\phi^{-1}(u)-sh),\phi(\phi^{-1}(v)-th)\right)-C\left(\phi(\phi^{-1}(u)-sh),v\right)\right]k(s,t)dsdt &\\
& -\int_{-1}^{1}\int_{-1}^{1}\left[C\left(\phi(\phi^{-1}(u)-sh),v\right)-C(u,v)\right]k(s,t)dsdt. &
\end{eqnarray*}
By applying successively a Taylor expansion for $C$ and for $\phi$, we get
\begin{eqnarray*}
\mathbb{E}g^2(U,V,h)&\leq &\int_{-1}^{1}\int_{-1}^{1}C_v\left(\phi(\phi^{-1}(u)-sh),\theta_1\right)\left[\phi(\phi^{-1}(v)-th)-\phi(\phi^{-1}(v))\right]k(s,t)dsdt\\
& & -\int_{-1}^{1}\int_{-1}^{1}C_u\left(\theta_2,v\right)\left[\phi(\phi^{-1}(u)-sh)-\phi(\phi^{-1}(u))\right]k(s,t)dsdt\\
&\leq & \int_{-1}^{1}\int_{-1}^{1}C_v\left(\phi(\phi^{-1}(u)-sh),\theta_1\right)\phi'(\gamma_1).(-th)k(s,t)dsdt\\
& & -\int_{-1}^{1}\int_{-1}^{1}C_u\left(\theta_2,v\right)\phi'(\gamma_2).(-sh)k(s,t)dsdt,
\end{eqnarray*}
where $\theta_1\in \left(\phi(\phi^{-1}(v)-th),v\right)\ ;\ \theta_2\in \left(\phi(\phi^{-1}(u)-sh),u\right)\ ;\ \gamma_1\in\left(\phi^{-1}(v)-th,\phi^{-1}(v)\right)\ ;\ \gamma_2\in\left(\phi^{-1}(u)-sh,\phi^{-1}(u)\right).$\\
This implies 
\begin{eqnarray*}
\mathbb{E}g^2(U,V,h)&\leq &4h\left\|C_v\right\|\left\|\phi'\right\|\left\|k\right\|^2\left|t\right|+4h\left\|C_u\right\|\left\|\phi'\right\|\left\|k\right\|\left|s\right|\\
&\leq & 4h\left\|\phi'\right\|\left\|k\right\|\left(\left\|C_v\right\|+\left\|C_u\right\|\right).
\end{eqnarray*}
Thus condition (G.ii) holds, with $C_0=4\left\|\phi'\right\|\left\|k\right\|\left(\left\|C_v\right\|+\left\|C_u\right\|\right).$\\

\textbf{Case 4.} $u\wedge\phi(\phi^{-1}(u)-sh)=u\ \text{et}\ v\wedge \phi(\phi^{-1}(v)-th)=\phi(\phi^{-1}(v)-th)$.\\ 
This case is analogous to Case 3, where the roles of $u$ and $v$ are interchanged. Hence, condition (G.ii) is fulfilled, with the same constant  $C_0=4\left\|\phi'\right\|\left\|k\right\|\left(\left\|C_v\right\|+\left\|C_u\right\|\right).$\\

\textbf{Checking for (F.i)}. We have to check the uniform entropy condition for the class of functions
$$
\mathcal{G}=\left\{\begin{array}{c} 
K\left(\frac{\phi^{-1}(u)-\phi^{-1}(\zeta_{1,n}(s))}{h}, \frac{\phi^{-1}(v)-\phi^{-1}(\zeta_{2,n}(t))}{h}\right)- \mathbb{I}\{s\leq u,t\leq v\},\\
 u,v\in[0,1], 0<h<1 \,\text{and }\, \zeta_{1,n}\zeta_{2,n}:[0,1]\mapsto [0,1] \,\text{nondecreasing.}
\end{array} \right\}
$$
 To this end, we consider the following classes of functions, where $\varphi$ is an increasing function :\\
\noindent
$ \mathbb{F}=\left\{(\varphi(x)+m)/\lambda,\lambda > 0,\;m\in\mathbb{R}\right\}$\\
$\displaystyle \mathbb{K}_0=\left\{K((\varphi(x)+m)/\lambda),\lambda>0, \;m\in\mathbb{R}\right\}$\\
$\displaystyle \mathbb{K}=\left\{K((\phi(x)+m)/\lambda,(\phi(y)+m)/\lambda), \lambda>0,\;m\in\mathbb{R}\right\}$\\
$\displaystyle \mathbb{H}=\left\{K((\phi(x)+m)/\lambda,(\phi(y)+m)/\lambda)-\mathbb{I}\left\{x\leq u,y\leq v\right\}\ ;\lambda> 0, \;m\in\mathbb{R}, (u,v)\in [0,1]^2\right\}$.\\

It is clear that by applying  lemmas 2.6.15 and 2.6.18 in van der Vaart and Wellner (see \cite{r9}, p. 146-147), the sets $\mathbb{F},\;\mathbb{K}_0,\;\mathbb{K},\;\mathbb{H}$ are all VC-subgraph classes. Thus, by choosing  the constant function ${\rm G}(x,y)\mapsto {\rm G}(x,y)=\left\|k\right\|^2+1 $ as an envelope function for the class $\mathbb{H}$ ( indeed ${\rm G}(x,y)\geq \sup_{g\in\mathbb{H}}\left|g(x,y)\right|,\ \forall (x,y))$, we can infer from Theorem 2.6.7 in \cite{r9} that  $\mathbb{H}$ satisfies the uniform entropy condition. Since $\mathbb{H}$ and $\mathcal{G}$ have the same structure, we can conclude that $\mathcal{G}$ satisfies this property too, i.e.
$$\exists \; C_0>0, \nu_0>0\ :\ N\left(\epsilon,\mathcal{G}\right)\leq C_0\epsilon^{-\nu_0},\quad 0<\epsilon<1.$$

\noindent \textbf{Checking for (F.ii).}\\
 Define the class of functions 
$$
\mathcal{G}_0=\left\{\begin{array}{c} 
K\left(\frac{\phi^{-1}(u)-\phi^{-1}(\zeta_{1}(s))}{h}, \frac{\phi^{-1}(v)-\phi^{-1}(\zeta_{2}(t))}{h}\right)- \mathbb{I}\{s\leq u,t\leq v\},\\
 u,v\in[0,1]\cap\mathbb{Q}, 0<h<1 \,\text{and }\, \zeta_{1}\zeta_{2}:[0,1]\mapsto [0,1] \,\text{nondecreasing.}
\end{array} \right\}
$$
It's clear that  $\mathcal{G}_0$ is countable and  $\mathcal{G}_0\subset\mathcal{G}$. Let
$$g(s,t)=K\left(\frac{\phi^{-1}(u)-\phi^{-1}(\zeta_1(s))}{h},\frac{\phi^{-1}(v)-\phi^{-1}(\zeta_2(t))}{h}\right)- \mathbb{I}\{s\leq u,s\leq v\}\in\mathcal{G}, (s,t)\in [0,1]^2$$
and for $m\geq 1$, 
$$g_m(s,t)=K\left(\frac{\phi^{-1}(u_m)-\phi^{-1}(\zeta_1(s))}{h} ,\frac{\phi^{-1}(v_m)-\phi^{-1}(\zeta_2(y))}{h}\right)- \mathbb{I}\{s\leq u_m,t\leq v_m\},$$
where $ u_m=\frac{1}{m^2}[m^2u]+\frac{1}{m^2}$ and $ v_m=\frac{1}{m^2}[m^2v]+\frac{1}{m^2}$.\\
\indent Let $\alpha_m=u_m-u,\quad \beta_m=v_m-v$. Then, we have  $\displaystyle 0<\alpha_m\leq\frac{1}{m^2}$ and $\displaystyle 0<\beta_m\leq\frac{1}{m^2}$. Hence $u_m\searrow u$ and $v_m\searrow v$. By continuity $\phi^{-1}(u_m)\searrow \phi^{-1}(u)$ and $\phi^{-1}(v_m)\searrow \phi^{-1}(v)$. Define
$$\delta_{m,u}=\left(\frac{\phi^{-1}(u_m)-\phi^{-1}(\zeta_1(x))}{h} \right)-\left(\frac{\phi^{-1}(u)-\phi^{-1}(\zeta_1(x))}{h} \right)=\frac{\phi^{-1}(u_m)-\phi^{-1}(u)}{h} $$ and
$$\delta_{m,v}=\left(\frac{\phi^{-1}(v_m)-\phi^{-1}(\zeta_2(y))}{h} \right)-\left(\frac{\phi^{-1}(v)-\phi^{-1}(\zeta_2(y))}{h} \right)=\frac{\phi^{-1}(v_m)-\phi^{-1}(v)}{h}.$$
Then
$\delta_{m,u}\searrow 0$ and $\delta_{m,v}\searrow 0$, 
which are equivalent to $$ \left(\frac{\phi^{-1}(u_m)-\phi^{-1}(\zeta_1(x))}{h} \right)\searrow \left(\frac{\phi^{-1}(u)-\phi^{-1}(\zeta_1(x))}{h} \right)$$
 and $$\left(\frac{\phi^{-1}(v_m)-\phi^{-1}(\zeta_2(y))}{h} \right)\searrow \left(\frac{\phi^{-1}(v)-\phi^{-1}(\zeta_2(y))}{h} \right).$$
By right-continuity of the kernel $K(\cdot,\cdot)$, we obtain
$$\forall (x,y)\in[0,1]^2, g_m(x,y)\longrightarrow g(x,y), m\rightarrow \infty$$
and conclude that $\mathcal{G}$ is pointwise measurable class.\\
\end{proof}

\end{document}